\documentclass[11pt]{amsart}
\usepackage{amssymb,color}
\numberwithin{equation}{section} 

\usepackage{graphicx}

\newcommand{\bea}{\begin{eqnarray}}
\newcommand{\eea}{\end{eqnarray}}
\newcommand{\ba}{\begin{array}}
\newcommand{\ea}{\end{array}}
\newcommand{\edc}{\end{document}}
\newcommand{\bc}{\begin{center}}
\newcommand{\ec}{\end{center}}
\newcommand{\be}{\begin{equation}}
\newcommand{\ee}{\end{equation}}

\newcommand{\dsf}{\displaystyle\frac}

\def\cb{{\mathcal B}}

\def\ce{{\mathcal E}}

\def\cg{{\mathcal G}}

\def\bc{{\mathbb C}}

\def\bn{{\mathbb N}}

\def\bq{{\mathbb Q}}
\def\br{{\mathbb R}}

\def\bz{{\mathbb Z}}

\def\a{\alpha}

  \def\G{\Gamma}

\def\l{\lambda} 
\def\k{\kappa}
\def\m{\mu}

\def\n{\nu}

\def\s{\sigma}

\def\w{\omega} \def\Om{\Omega}

\def\h{{\mathbf{h}}}

\def\sb{{\mathbf{s}}}

\newtheorem{thm}{Theorem}[section]
\newtheorem{lem}[thm]{Lemma}
\newtheorem{cor}[thm]{Corollary}
\newtheorem{prop}[thm]{Proposition}

\theoremstyle{remark}
\newtheorem{rem}{Remark}[section]

\date{\today}
\begin{document}

\title[On $P$-adic dynamical systems]
{On chaotic behavior of the $P$-adic generalized Ising mapping and
its application}


\author{Farrukh Mukhamedov}
\address{Farrukh Mukhamedov\\
Department of Mathematical Sciences\\
College of Science, The United Arab Emirates University\\
P.O. Box, 15551, Al Ain\\
Abu Dhabi, UAE} \email{{\tt far75m@gmail.com} {\tt
farrukh.m@uaeu.ac.ae}}

\author{Hasan Ak\i n}
\address{Hasan Ak\i n, Ceyhun Atuf Caddesi 1164. Sokak 9/ 4
╟ankaya, Ankara, Turkey.} \email{{\tt akinhasan25@gmail.com}}

\author{Mutlay Dogan}
\address{Mutlay Dogan, Department of Mathematics, Ishik University, Iraq}

\begin{abstract}

In the present paper, by conducting research on the dynamics of the
$p$-adic generalized Ising mapping corresponding to renormalization
group associated with the $p$-adic Ising-Vannemenus model on a
Cayley tree, we have determined the existence of the fixed points of
a given function. Simultaneously, the attractors of the dynamical
system have been found. We have come to a conclusion that the
considered mapping is topologically conjugate to the symbolic shift
which implies its chaoticity and as an application, we have
established the existence of periodic $p$-adic Gibbs measures for
the $p$-adic Ising-Vannemenus model.

\vskip 0.3cm \noindent {\it
Mathematics Subject Classification}: 46S10, 82B26, 12J12, 39A70, 47H10, 60K35.\\
{\it Key words}: $p$-adic numbers; $p$-adic dynamical system; chaos;
periodic.
\end{abstract}

\maketitle

\section{introduction}

In \cite{Egg} the thermodynamic behavior of the central site of an
Ising spin system with ferromagnetic nearest-neighbor interactions
on a Cayley tree was studied by recursive methods which consequently
opened new perspectives between the recursion approach and the
theory of dynamical systems. The existence of the phase transition
is closely connected to the existence of the chaotic behavior of the
associated dynamical system which is governed by the Ising-Potts
function. Investigating the dynamics of this function has been the
object of no small amount of study in the real and complex settings.
This deceptively simple family of rational functions has given rise
to a surprising number of interesting dynamical features (see for
example \cite{BG,FTC,Kap,Monr}). Therefore, the combination of
statistical mechanics tools and methods adopted from dynamical
systems are one of the most promising directions in the theory of
phase transitions. One of such tools is the renormalization group
(RG) which has had a profound impact on modern statistical physics.
This method appeared after WilsonТs seminal work in the early 1970's
\cite{Wil}, based also on the ground breaking foundations laid by
Kadanoff, Widom, Michael Fisher \cite{Fish}.

On the other hand, there are many investigations that have been
conducted to discuss and debate the question due to the assumption
that $p$-adic numbers provide a more exact and more adequate
description of microworld phenomena \cite{Kh1,VVZ}) Consequently,
various models in physics described in the language of $p$-adic
analysis (see for example, \cite{ACK,ADFV})), and numerous
applications of such an analysis to mathematical physics have been
studied in \cite{Kh1,KKOJ,KOJ}. These investigations proposed to
study new probability models (namely $p$-adic probability), which
cannot be described using ordinary Kolmogorov's probability theory
\cite{KL}. Using that, $p$-adic measure theory in \cite{M13,M15,RK2}
the theory of $p$-adic statistical mechanics has been been
developed. For complete review of the $p$-adic mathematical physics
we refer to \cite{Drag}.

On the  other hand, recently, polynomials and rational maps of
$p$-adic numbers $\bq_p$ have been studied as dynamical systems over
this field \cite{B0,B1}. It turns out that these $p$ -adic dynamical
systems are quite different to the dynamical systems in Euclidean
spaces (see for example, \cite{AKh,FL2,HY,KhN,L,Sil1} and their
bibliographies therein). In theoretical physics, the interest in
$p$-adic dynamical systems was started with the development of
$p$-adic models \cite{M13,MK16,KhN}. In these investigations, the
importance of detecting chaos was stressed in the $p$-adic setting
\cite{Kh22,TVW,Wo}. In \cite{M15} the renormalization group method
has been developed to study phase transitions for several $p$-adic
models on Cayley trees. In \cite{MK16} we have studied some
particular cases of the Ising-Potts function and showed its
chaoticity.

So presently, we study dynamics of a $p$-adic rational
mapping\footnote{We note that the dynamical properties of the fixed
points of some $p$-adic rational maps have been studied in
\cite{B0,B1,FL4,KM1,RL,RS}. Some results on  the global structure of
rational maps on $\bq_p$ can be found in \cite{DS11,Yu}.} (see
\eqref{gg}) \textit{the generalized Ising mapping} associated with
the Ising-Vannimenus model on a Cayley tree \cite{MDA}. The
existence of the phase transition for this model has been
investigated in \cite{MDA,MSK1}. In section 3 we study the existence
of fixed points of the mentioned mapping. In section 4 the dynamical
behavior of the fixed points is explored and particularly, the basin
of attraction of the attractive fixed point is described. In section
5, we show that the dynamical system \eqref{gg} is topologically
conjugate to the full shift, and hence it is chaotic. The obtained
result opens certain perspectives in the study of generalized
self-similar sets in the $p$-adic setting \cite{MK161}.  As an
application of the result of section 5, in the last section 6, we
illustrate the existence of periodic $p$-adic Gibbs measures for the
Ising-Vannimenus model. Note that a construction of 2-periodic
$p$-adic Gibbs measures has been given in \cite{MSK1}, but using
that construction, it is extremely difficult to find other kinds of
periodic measures. Note that, in the $p$-adic setting, due to the
lack of convex structure of the set of $p$- adic Gibbs measures, it
is quite difficult to constitute a phase transition with some
features of the set of $p$-adic Gibbs measures. The result of
section 6 implies that the set of $p$-adic Gibbs measures is huge.
Moreover, the advantage of the present work allows us to find lots
of periodic Gibbs measures. We point out that some numerical
simulations predict \cite{V} the chaotic behavior of the mapping in
the real setting, but these are no rigorous proofs of this kind of
fact.

\section{Preliminaries}
\subsection{$p$-adic numbers}
In what follows $p$ will be a fixed prime number, and by $\bq_p$ it is denoted the field of
$p-$adic numbers, which is a completion of the rational numbers
$\bq$ with respect to the norm $|\cdot|_p:\bq\to\br$
given by
\begin{eqnarray}\label{p-norm}
|x|_p=\left\{
\begin{array}{c}
  p^{-r} \ x\neq 0,\\
  0,\ \quad x=0,
\end{array}
\right.
\end{eqnarray}
here, $x=p^r\frac{m}{n}$ with $r,m\in\bz,$ $n\in\bn$,
$(m,p)=(n,p)=1$. The number $r$ is called \textit{the $p-$order} of
$x$ and it is denoted by $ord_p(x)=r.$ The absolute value
$|\cdot|_p$, is non-Archimedean, meaning that it satisfies the
ultrametric triangle inequality $|x + y|_p \leq \max\{|x|_p,
|y|_p\}$.

Any $p$-adic number $x\in\bq_p$, $x\neq 0$ can be uniquely represented in the form
\begin{equation}\label{canonic}
x=p^{ord_p(x)}(x_0+x_1p+x_2p^2+...),
\end{equation}
where $x_j$ are integers such that $0\leq x_j\leq p-1$ for all
$j=0,1,2,\dots$ and $x_0>0$,
 In this case $|x|_p=p^{-ord_p(x)}$.


\begin{lem}\label{sqr}
Let $p\geq 3$, $a\in \mathbb{Q}_{p}$ and
$a=p^{ord_p(a)}(a_{0}+a_{1}p+a_{2}p^{2}+...)$. Then $\sqrt{a}$
exists in $\mathbb{Q}_{p}$ if and only if

i) $ord_p(a)\in2\mathbb{Z}_{p}$

ii) $x^{2}\equiv a_{0}(mod p)$ has a solution in $\mathbb{Z}$.
\end{lem}

In what follows, for the sake of simplicity, we always assume that
$p\geq 3$, since $p=2$ is considered as a pathological case.

For each $a\in \bq_p$, $r>0$ we denote
$$ B_r(a)=\{x\in \bq_p : |x-a|_p< r\}, \ \  S_r(a)=\{x\in\bq_p:\
|x-a|_p=r\}$$ and the set of all {\it $p$-adic integers}
$$\bz_{p}=\left\{ x\in \bq_{p}:\
|x|_{p}\leq1\right\}.$$ The set $\bz_p^*=\bz_p\setminus p\bz_p$ is
called a set of $p$-adic units. Recall that the $p$-adic logarithm
is defined by the series
$$
\log_p(x)=\log_p(1+(x-1))=\sum_{n=1}^{\infty}(-1)^{n+1}\dsf{(x-1)^n}{n},
$$
which converges for every $x\in B_1(1)$. And the $p$-adic
exponential is defined by
$$
\exp_p(x)=\sum_{n=0}^{\infty}\dsf{x^n}{n!},
$$
which converges for every $x\in B_{p^{-1/(p-1)}}(0)$. Note that, in
the considered setting (i.e. $p\geq 3$), due to the discreteness of
the norm, we have $B_{p^{-1/(p-1)}}(0)=B_1(0)$.

\begin{lem}\label{21} (\cite{Ko},\cite{VVZ})  Let $x\in
B_1(0)$ then we have $$ |\exp_p(x)|_p=1,\ \ \
|\exp_p(x)-1|_p=|x|_p<1, \ \ |\log_p(1+x)|_p=|x|_p<p^{-1/(p-1)} $$
and $$ \log_p(\exp_p(x))=x, \ \ \exp_p(\log_p(1+x))=1+x. $$
\end{lem}

%
Denote
\begin{equation}\label{Exp}
\ce_p=\{x\in\bq_p: \  |x-1|_p<1\}.
\end{equation}
Using Lemma \ref{21} one can prove the following facts.

\begin{lem}\label{MMK1}\cite{MSK1}
The set $\ce_p$ has the following properties:
\begin{enumerate}
  \item[(1)] $\ce_p$ is a group under multiplication;
  \item[(2)] $|a-b|_p<1$ for all $a,b\in \ce_p$;
  \item[(3)] if $a,b\in \ce_p$ then it holds $|a+b{{|}_{p}}=1$;
 \item[(4)] if $a\in \ce_p$, then there is an element $h\in B(0,p^{-1/(p-1)})$ such that $a=\exp_p(h);$
 \item[(5)] if $a\in \ce_p$, then $\sqrt{a}\in \ce_p$.
\end{enumerate}
\end{lem}

Note that the basics of $p$-adic analysis, $p$-adic mathematical
physics are explained in \cite{Kh1,Ko}.

\subsection{Dynamical systems in $\bq_p$}
In this subsection we recall some standard terminology of the theory
of dynamical systems (see for example \cite{KhN}).

A function $f:B_r(a)\to\bq_p$ is said to be {\it analytic} if it can
be represented by
$$
f(x)=\sum_{n=0}^{\infty}f_n(x-a)^n, \ \ \ f_n\in \bq_p,
$$ which converges uniformly on the ball $B_r(a)$.

Consider a dynamical system $(f,B)$ in $\bq_p$, where $f: B\to B$ is
an analytic function and $B=B_r(a)$ or $\bq_p$. Denote
$x^{(n)}=f^n(x^{(0)})$, where $x^{(0)}\in B$ and
$f^n(x)=\underbrace{f\circ\dots\circ f(x)}_n$.
 If $f(x^{(0)})=x^{(0)}$ then $x^{(0)}$
is called a {\it fixed point}. A fixed point $x^{(0)}$ is called an
{\it attractor} if there exists a neighborhood $U(x^{(0)})(\subset
B)$ of $x^{(0)}$ such that for all points $y\in U(x^{(0)})$ it holds
that $\lim\limits_{n\to\infty}y^{(n)}=x^{(0)}$, where
$y^{(n)}=f^n(y)$. If $x^{(0)}$ is an attractor then its {\it basin
of attraction} is
\begin{eqnarray}\label{attracting}
A(x^{(0)})=\{y\in \bq_p :\ y^{(n)}\to x^{(0)}, \ n\to\infty\}.
\end{eqnarray}
A fixed point $x^{(0)}$ is called {\it repeller} if there  exists
a neighborhood $U(x^{(0)})$ of $x^{(0)}$ such that
$|f(x)-x^{(0)}|_p>|x-x^{(0)}|_p$ for $x\in U(x^{(0)})$, $x\neq
x^{(0)}$.

Let $x^{(0)}$ be a fixed point of an analytic function $f(x)$. Set
$$
\l=\frac{d}{dx}f(x^{(0)}).
$$

The point $x^{(0)}$ is called {\it attracting} if $0\leq |\l|_p<1$,
{\it indifferent} if $|\l|_p=1$, and {\it repelling} if $|\l|_p>1$.

\subsection{$p$-adic sub-shift}

Let $f:X\to\mathbb Q_p$ be a map from a compact open set $X$ of $\mathbb Q_p$ into $\mathbb
Q_p$. We assume that (i) $f^{-1}(X)\subset X$; (ii) $X=\cup_{j\in I}B_{r}(a_j)$ can be written as a finite disjoint union of balls
of centers $a_j$ and of the same radius $r$ such that for each $j\in I$ there is an integer $\tau_j\in\mathbb Z$ such that
\begin{equation}\label{tau}
|f(x)-f(y)|_p=p^{-\tau_j}|x-y|_p,\ \ \ \ x,y\in B_r(a_j).
\end{equation}
For such a map $f$, define its Julia set by
\begin{equation}\label{J}
J_f=\bigcap_{n=0}^\infty f^{-n}(X).
\end{equation}
It is clear that $f^{-1}(J_f)=J_f$ and then $f(J_f)\subset J_f$.

Following \cite{FL2} the triple $(X,J_f,f)$ is called a $p$-adic
{\it weak repeller} if all $\tau_j$ in \eqref{tau}  are nonnegative,
but at least one is positive. We call it a $p$-adic {\it repeller}
if all $\tau_j$ in \eqref{tau} are positive.
 For any $i\in I$, we let
$$
I_i:=\left\{j\in I: B_r(a_j)\cap
f(B_r(a_i))\neq\varnothing\right\}=\{j\in I: B_r(a_j)\subset
f(B_r(a_i))\}
$$
(the second equality holds because of the expansiveness and the
ultrametric property). Then define a matrix $A=(a_{ij})_{I\times
I}$, called \textit{incidence matrix} as follows
$$
a_{ij}=\left\{\begin{array}{ll}
1,\ \ \mbox{if }\ j\in I_i;\\
0,\ \ \mbox{if }\ j\not\in I_i.
\end{array}
\right.
$$
If $A$ is irreducible, we say that $(X,J_f,f)$ is
\textit{transitive}. Here the irreducibility of $A$  means, for
any pair $(i,j)\in I\times I$ there is positive integer $m$ such
that $a_{ij}^{(m)}>0$, where $a_{ij}^{(m)}$ is the entry of the
matrix $A^m$.

Given $I$ and the irreducible incidence matrix $A$ as above, we
denote
$$
\Sigma_A=\{(x_k)_{k\geq 0}: \ x_k\in I,\  A_{x_k,x_{k+1}}=1, \
k\geq 0\}
$$
which is the corresponding subshift space, and let $\sigma$ be the
shift transformation on $\Sigma_A$. We equip $\Sigma_A$ with a
metric $d_f$ depending on the dynamics which is defined as follows.
First for $i,j\in I,\ i\neq j$ let $\k(i,j)$ be the integer such
that $|a_i-a_j|_p=p^{-\k(i,j)}$. It clear that $\k(i,j)<-\log_p(r)$,
where $r$ is the radius of the balls at the beginning of section
2.3. By the ultra-metric inequality, we have
$$
|x-y|_p=|a_i-a_j|_p\ \ \ i\neq j,\ \forall x\in B_r(a_i), \forall
y\in B_r(a_j)
$$
For $x=(x_0,x_1,\dots,x_n,\dots)\in\Sigma_A$ and
$y=(y_0,y_1,\dots,y_n,\dots)\in\Sigma_A$, define
$$
d_f(x,y)=\left\{\begin{array}{ll}
p^{-\tau_{x_0}-\tau_{x_1}-\cdots-\tau_{x_{n-1}}-\k(x_{n},y_{n})}&, \mbox{ if }n\neq0\\
p^{-\k(x_0,y_0)}&, \mbox{ if }n=0
\end{array}\right.
$$
where $n=n(x,y)=\min\{i\geq0: x_i\neq y_i\}$. It is clear that
$d_f$ defines the same topology as the classical metric which is
defined by $d(x,y)=p^{-n(x,y)}$.

\begin{thm}\label{exit}(\cite{FL2}) Let $(X,J_f,f)$ be a transitive $p$-adic weak repeller with incidence matrix $A$.
Then the dynamics $(J_f,f,|\cdot|_p)$ is isometrically conjugate to
the shift dynamics $(\Sigma_A,\sigma,d_f)$.
\end{thm}

\section{Description and existence of the fixed points}

In this section we consider a dynamical system corresponding to a
non-linear function $f_{a,b}:\mathbb{Q}_p\to \mathbb{Q}_p$ given by
\begin{equation}\label{case1}
f_{a,b}(u)=\left(\frac{(abu)^{2}+1}{b^{2}+a^{2}u^{2}}\right), \ \ a,b\in \ce_p.
\end{equation}

We stress that the mentioned function has recently been studied in
\cite{MDA} to find $p$-adic Gibbs measures. To investigate other
properties of the Gibbs measures it is needed to study dynamics of
$f_{a,b}$.

First, recall that two functions $f$ and $g$ are
\textit{topologically conjugate} if there exists a homeomorphism $h$
such that $h\circ f=g\circ h$. Moreover, if $h$ is isometry, then
the functions $f$ and $g$ are called \textit{isometrically
conjugate}. It is clear that the topological conjugacy is weak that
isometrical conjugacy, since the function $h$ is not necessary to be
an isometry in general.

One can prove the following fact.

\begin{prop}
The function $f_{a,b}$ is topologically  conjugate to
\begin{equation}\label{gg}
g_{a,b}(u)=a\left(\frac{b^2u^{2}+1}{b^{2}+u^{2}}\right).
\end{equation}
\end{prop}

Note that the conjugation in the last proposition is isometric
conjugacy.

 The function \eqref{gg} is called \textit{a generalized
Ising mapping}. In what follows, as we mentioned, we always assume
$p\geq 3$ without stressing it.

We will investigate the function $g_{a,b}.$ Some proofs of the
relations between the coefficients $a,b$ and the fixed points of
$g_{a,b}$ are given in \cite{MDA}. For the sake of completeness, we
are going to prove auxiliary facts.

\begin{prop}\label{parti2}
Let $p\geq 3$ and $a,b\in \ce_p$. Then the following statements are
true:
\begin{enumerate}
  \item[(i)] $g_{a,b}(\ce_p)\subset \ce_p$,
  \item[(ii)]$g_{a,b}$ is a contraction on $\ce_p$,
\end{enumerate}
\end{prop}
\begin{proof} (i) To prove this let us examine $| g_{a,b}(u)-1|_{p}\leq\frac{1}{p},$ for all $u\in \ce_p.$
Indeed, from $a,b\in\ce_{p}$, $|ab^{2}-1|_{p}<1$, $|b^{2}-1|_{p}<1$,
and $|2|_{p}=1$, one gets
\begin{eqnarray*}
\mid g_{a,b}(u)-1\mid_{p}&=&\frac{|ab^{2}u^{2}+a-b^{2}-u^{2}|_p}{|b^{2}+u^{2}|_p}\\
&=&\mid(ab^{2}-1)u^{2}+a-1-(b^{2}-1)\mid_{p}\\
&\leq&\frac{1}{p}.
\end{eqnarray*}
(ii) Now we show that the function $g_{a,b}$ is a contraction on
$\ce_{p}$. Indeed,
\begin{eqnarray*}
|g_{a,b}(x)-g_{a,b}(y)|_{p}&=&\frac{|(b^4-1) (x^2-y^2)|_p}{|(b^2+x^2)(b^2+y^2)|_p}\\
&=&\mid b^{4}-1\mid_{p}\mid x^{2}-y^{2}\mid_{p}\\
&\leq&\frac{1}{p}\mid x-y\mid_{p}
\end{eqnarray*}
\end{proof}

From this proposition due to the Banach Contraction Principle we
infer that there exists $x_{0}\in\ce_{p}$ such that
$g_{a,b}(x_{0})=x_{0}$.

Now we are going to describe all fixed points of $g_{a,b}.$

\begin{thm}\label{g-fix}
Let $a,b\in\ce_{p}$ with $b\neq 1$, and $g_{a,b}$ be given by
\eqref{gg}. Then the following statements hold:
\begin{enumerate}
\item[(i)] the function $g_{a,b}$ has a unique fixed point $x_{0}$ in
$\ce_{p}$;
\item[(ii)] if $p\equiv 3 (mod \ 4)$, then $x_0$ is the only fixed point of $g_{a,b}$. If $p\equiv 1 (mod \ 4)$, then  $g_{a,b}$ has at
two fixed points $x_{1},x_{2}$ different from $x_0$;

\item[(iii)] let $x_{1},x_{2}$ be two fixed points of $g_{a,b}$ different from
$x_0$. Then $x_{1},x_{2}\in\bq_p\setminus\ce_p$.
\end{enumerate}
\end{thm}

\begin{proof} (i) By Proposition \ref{parti2} we conclude that $g_{a,b}$ satisfies the Banach
Contraction principle on $\ce_p$. Therefore, $g_{a,b}$ has a unique
fixed point belonging to $\ce_p$, i.e. there exists $x_0\in\ce_p$
such that $g_{a,b}(x_0)=x_0.$

(ii) Consider the equation $x=g_{a,b}(x)$, which can be rewritten as
follows
\[
x^{3}-ab^{2}x^{2}+b^{2}x-a=0
\]
Note that, in general, we may solve the last equation by methods
developed in \cite{MOS}, but those methods give only information
about the existence of solutions. In reality, we need more
properties of the solutions (see further sections). Therefore, we
are going to find all the solutions.

Due to (i) we know that one of solutions of the last equation is
$x_{0}$. Therefore, one can write

\begin{equation}\label{canonicU}
x^{3}-ab^{2}x^{2}+b^{2}x-a=(x-x_{0})(x^{2}+(x_{0}-ab^{2})x+(x_{0}^2-ab^{2}x_0+b^{2})).
\end{equation}

The equality \eqref{canonicU} yields that
$$\frac{a}{x_{0}}=x_{0}^2-ab^{2}x_0+b^{2}.$$

Hence, the quadratic equation of RHS of \eqref{canonicU} can be
rewritten as follows
\begin{equation}\label{quadratic1}
x^{2}+(x_{0}-ab^{2})x+\frac{a}{x_{0}}=0.
\end{equation}
For us it is enough to solve \eqref{quadratic1}. Let us find its
discriminant
\begin{eqnarray*}
\triangle&=&x_{0}^{2}-2ab^{2}x_{0}+a^{2}b^{4}-4\frac{a}{x_{0}}\\
&=&x_{0}^{2}-2ab^{2}x_{0}+a^{2}b^{4}-4(x_{0}^{2}-ab^{2}x_{0}+b^{2})\\
&=&-3 x_0^2 +2 a b^2 x_0-4 b^2+a^2 b^4,
\end{eqnarray*}
here we have used the equality
$\frac{a}{x_{0}}=x_{0}^{2}-ab^{2}x_{0}+b^{2}$.



By using simple calculations, we find
\begin{eqnarray}\label{canonic 6}
\triangle&=&-4-4(x_0-1)+4(a-1)(b^2-1)+4(a-1)\\\nonumber
&+&2(x_0-1)(ab^2-1)-3(x_0-1)^2+(ab^2-1)^2.
\end{eqnarray}
Now, taking into account $a,b,x_0\in\ce_p$, one gets
\begin{eqnarray}\label{Delta1}
\triangle=-4+p^n\delta.
\end{eqnarray}
According to Lemma \ref{sqr} we conclude that $\sqrt{\triangle}$
exists if and only if $\sqrt{-4}$ exists, which is equivalent the
existence of $\sqrt{-1}$. Taking into account the fact that $\sqrt{-1}$ exists in $\mathbb{Q}_p$ if and only if $p\equiv 1 (mod \ 4)$.

(iii) Assume that \eqref{quadratic1} has two solutions
$x_{1},x_{2}$. So Viete's theorem implies that
$$\begin{cases}
x_{1}+x_{2}=ab^{2}-x_{0}\\
x_{1}\cdot x_{2}=\frac{a}{x_{0}}
\end{cases}$$
From $\mid x_{0}-ab^{2}\mid_{p}\leq\frac{1}{p}$ we get $\mid
x_{1}\mid_{p}\mid x_{2}\mid_{p}=1,\mid
x_{1}+x_{2}\mid_{p}\leq\frac{1}{p}.$ Hence, we obtain $\mid x_{1}\mid_{p}=1$, $\mid x_{2}\mid_{p}=1$.

Note that the solutions have the following form
\begin{eqnarray}\label{solx}
x_{1,2}=\frac{ab^{2}-x_{0}\pm\sqrt{\triangle}}{2}.
\end{eqnarray}

From \eqref{canonic 6} one can see that
\begin{eqnarray}\label{canonic61}
|(\sqrt{\triangle}-2)(\sqrt{\triangle}+2)|_{p}&=&
|\triangle-4|_{p}\nonumber\\[2mm]
&=&|-8+p^{-\gamma_{3}}\varepsilon_{3}|_{p}\\\nonumber
&=&1.\nonumber
\end{eqnarray}

Hence, one gets that
\begin{equation}\label{canonic 9}
|\sqrt{\triangle}-2|_{p}=1,\ |\sqrt{\triangle}+2|_{p}=1.
\end{equation}
Therefore, from \eqref{canonic 9} together with $|ab^{2}-1|_{p}<1$,
$|x_{0}-1|_{p}<1$ we have
\begin{eqnarray}\label{canonic 10}
|x_{1,2}-1|_{p}&=&|ab^{2}-x_{0}\pm\sqrt{\triangle}-2|_{p}\nonumber\\[2mm]
&=&|ab^{2}-1-(x_{0}-1)\pm\sqrt{\triangle}-2|_{p}\\
&=&1\nonumber
\end{eqnarray}
which means $x_{1,2}\notin\ce_{p}.$ This completes the proof.
\end{proof}


\begin{lem}\label{pr-g}
Let $a,b\in \ce_p$ with $b\neq 1$, and $g_{a,b}$ be given by
\eqref{gg}. Assume that $x_{0},x_{1},x_{2}$ are fixed points of
$g_{a,b}$. Then the following statements hold:
\begin{item}{\label{center}}
\begin{enumerate}
\item[(i)] ${\left|x_{0}-a\right|_{p}}<\left|b-1\right|_{p}$;
\item[(ii)]
$|b^{2}-1+(b^{2}a-x_{0})x_{1,2}|_p=\left|b-1\right|_{p}$;
\item[(iii)]
$\left|x^{2}_{1,2}+b^{2}\right|_{p}=\left|b-1\right|_{p}$;
\item[(iv)]
$\left|x^{2}_{0}+b^{2}\right|_{p}\left|x^{2}_{0}b^{2}+1\right|_{p}=1$;
\item[(v)] $\left|x^{2}_{1,2}+b^{2}\right|_{p}\left|x^{2}_{1,2}b^{2}+1\right|_{p}=|x^{2}_{1,2}+b^{2}|^2_{p}\leq\frac{1}{p^{2}}$.
\item[(vi)] $|x_0-1|_p<|b-1|_p$, if $|a-1|_p<|b-1|_p$;
\item[(vii)] if $|a-1|_p<|b-1|_p$, then
$$
\triangle=-4+p^{2m+l}\delta,
$$
where $b-1=p^{m}\varepsilon$, $l\in \mathbb{N}$ and $\delta\in \mathbb{Z}_p.$
\end{enumerate}
\end{item}
\end{lem}

\begin{proof} (i) From Lemma \ref{MMK1} we find
\begin{eqnarray*}
\mid x_{0}-a\mid_{p}&=&\mid g(x_{0})-a\mid_{p}\\[2mm]
&=&\frac{|b^{2}x_{0}^{2}+1-b^{2}-x_{0}^{2}|_p}{|b^{2}+x_{0}^{2}|_{p}}\\[2mm]
&=&| x_{0}^{2}-1|_{p}| b-1|_{p}\\[2mm]
&=&| x_{0}-1|_{p}| b-1|_{p}
\end{eqnarray*}
Due to $|x_0-1|_p<1$ the last expression implies the desired
inequality.

(ii) We first observe that from \eqref{canonic 9} similarly to
\eqref{canonic 10} one can get $\left|x_{1,2}+1\right|_{p}=1$. This
together with the strong triangle inequality yields that
\begin{equation}\label{canonic 11}
\left|ax_{1,2}+1\right|_{p}=\left|1+x_{1,2}+(a-1)x_{1,2}\right|_{p}=1
\end{equation}
here we have used that $|a-1|_p<1$.

Now taking into account (i) and \eqref{canonic 11} we obtain
\begin{eqnarray*} |{b^{2}-1+(b^{2}a-x_{0})x_{1,2}}|_{p}&=&|{b^{2}-1+(b^{2}-1)ax_{1,2}-(x_{0}-a)x_{1,2}}|_{p}\\
&=&|{(b^{2}-1)(1+ax_{1,2}-(x_{0}-a)x_{1,2})}|_{p}\\
&=&| b^{2}-1|_{p}\\
&=&|b-1|_{p}.
\end{eqnarray*}

(iii)  Due to $x^{2}_{1,2}=(ab^{2}-x_{0})x_{1,2}-\frac{a}{x_{0}}$
and taking into account (i) and (ii) one finds
\begin{eqnarray*}
|b^{2}+x_{1,2}^{2}|_{p}&=&|
b^{2}+(ab^{2}-x_{0})x_{1,2}-\frac{a}{x_{0}}|_{p}\\[2mm]
&=&|
b^{2}-1+(ab^{2}-x_{0})x_{1,2}-\frac{1}{x_{0}}(a-x_{0})|_{p}\\[2mm]
&=&|b-1|_{p}.
\end{eqnarray*}

(iv) The proof of (iv) immediately follows from Lemma \ref{MMK1}.

(v) Now the fact that $x_{1,2}$ is a fixed point of $g$ together
with (iii) we obtain
\begin{eqnarray*}
\left|x_{1,2}^{2}+b^{2}\right|_p\left|x_{1,2}^{2}b^{2}+1\right|_p
&=&\left|x_{1,2}^{2}+b^{2}\right|^{2}_p\frac{|x_{1,2}^{2}b^{2}+1|_p}{|x_{1,2}^{2}+b^{2}|_p}\\[2mm]
&=&\left|x_{1,2}^{2}+b^{2}\right|^{2}_p\left|x_{1,2}\right|_p\\[2mm]
&=&\left|x_{1,2}^{2}+b^{2}\right|_{p}^{2}\\[2mm]
&\leq&\frac{1}{p^{2}}.
\end{eqnarray*}

(vi)
\begin{eqnarray*}
\left|x_{0}-1\right|_p&=&\left|g_{a,b}(x_{0})-1\right|_p\\
&=&\bigg|\frac{a(b^{2}x_0^{2}+1)}{b^{2}+x_0^{2}}-1\bigg|_p\\
&=&|a(b^{2}x_0^{2}+1)-(b^{2}+x_0^{2})|_p\\
&=&|a(b^{2}-1)(x_0^{2}-1)+(a-1)(x_0^2-1)+a(b^{2}-1)+2(a-1)+(1-b^2)|_p\\
&=&\left\{ \begin{array}{*{35}{l}}
   |a-1{{|}_{p}}, & \text{if}\ |a-1{{|}_{p}}\geq |b-1{{|}_{p}};  \\
   |b-1{{|}_{p}}{{p}^{-l}}, & \text{if}\ |a-1{{|}_{p}}<|b-1{{|}_{p}},\ l\geq 1.  \\
\end{array} \right.
\end{eqnarray*}

(vii) Using the proof Lemma \ref{pr-g} (i) and $|a-1|_p<|b-1|_p$ we
find $|ab^2-1|_p=|b-1|_p$. This together with \eqref{canonic 6} and
Lemma \ref{pr-g} (vi) implies the required assertion.

This completes the proof.
\end{proof}

\section{The classification of the fixed points of $p$-adic dynamical system}

In this section, we study behavior of the fixed points of function \eqref{gg}.
We will describe the behavior of the function \eqref{gg} with respect to the parameters $a,b\in \mathbb{Q}_p$
whether the fixed points are attracting, neutral or repelling.
\begin{thm}\label{theorem-fpt1}
Let $a,b\in\ce_{p}$ and $g_{a,b}$ has three  fixed points $x_{0},\,
x_{1},\, x_{2}$. Then the following statements hold:
\begin{enumerate}
\item[(i)] $x_{0}$ is an attracting fixed point;
\item[(ii)] $x_{1}$ and $x_{2}$ are repelling fixed points.
\end{enumerate}
\end{thm}
\begin{proof}
From
$$
\bigg| \frac{dg_{a,b}}{dx}(x_{0})\bigg|_{p}=\frac{|2a|_p\mid
x_{0}\mid_{p}\mid b^{4}-1\mid_{p}}{\mid
b^{2}+x_{0}^{2}\mid_{p}^{2}}=\mid b^{4}-1\mid_{p}\leq \frac{1}{p}
$$
We conclude that $x_{0}$ is attracting.

One can calculate that
$$\bigg| \frac{dg_{a,b}}{dx}(x_{1,2})\bigg|_{p}=\frac{|2a|_p| x_{1,2}|_{p}|b^{4}-1|_{p}}{| b^{2}+x_{1,2}^{2}|_{p}^{2}}=\frac{|b^{4}-1|_{p}}{| b^{2}+x_{1,2}^{2}|_{p}^{2}}=\frac{| b-1|_{p}}{| b^{2}+x_{1,2}^{2}|_{p}^{2}},
$$
here we have used $| b^{4}-1|_{p}=| b^{2}+1|_{p}|
b-1|_{p}|b+1|_{p}=| b-1|_{p}$. From Lemma \ref{pr-g} (iii), we find
\[
\bigg| \frac{dg_{a,b}}{dx}(x_{1,2})\bigg|_{p}=\frac{\mid
b-1\mid_{p}}{\mid b-1\mid_{p}^{2}}=\frac{1}{\mid b-1\mid_{p}}\geq p
\]
As a conclusion one gets that $x_{1}$ and $x_{2}$ are repelling fixed points.
\end{proof}

Now we are going to describe the basin of attraction
\[
A(x_{0})=\left\{ x\in \bq_{p}:\ g_{a,b}^n(x)\rightarrow
x_{0}\right\}
\]
of the fixed point $x_{0}$.

Now we define
\begin{eqnarray}\label{KR}
&& K=\{x\in S_1(x_0): |x^2+1|_p\leq |b^2-1|_p\},\\[2mm]
&& R=\bigcap_{n=0}g_{a,b}^{-1}(K).
\end{eqnarray}

Now, we are going to describe the size of the attractor of the
dynamic system.

\begin{thm}\label{theorem-fpt2}
Let $a,b\in \ce_p$. Then one has
$$
A(x_{0})=\bq_p\setminus R.
$$
Note that $R$ is nonempty if and only if $p\equiv 1 (mod \ 4)$.
\end{thm}

\begin{proof} Let us break the proof into three steps.

(I). According to Proposition \ref{parti2} for any $x\in\ce_{p}$ we
infer that $ x\in A(x_{0})$ which means $\ce_{p}\subset A(x_{0})$.
We notice that $\ce_p=B_1(x_0).$

(II) In this step we establish that if $x\notin\bq_p\setminus R$,
then $x\in\ce_p$. To do this, let us first assume $x\notin K$, then
show $x\in\ce_p$. Indeed, from the assumption we infer that
$|x^2+1|_p>|b^2-1|_p$. Therefore, one gets
\begin{align*}
| g_{a,b}(x)-1|_{p} & =\frac{|(ab^{2}-1)x^{2}+a-b^{2}|_{p}}{|x^{2}+b^{2}|_{p}}\\
& =\frac{|(ab^{2}-1)(x^{2}+1)-(a+1)(b^{2}-1)|_{p}}{| x^{2}+1+b^{2}-1\mid_{p}}\\
 & \leq\frac{\max\left\{ |ab^{2}-1|_{p}|x^{2}+1|_{p},| b^{2}-1|_{p}\right\} }{| x^{2}+1|_{p}}\\
 & <\frac{|x^{2}+1|_{p}}{| x^{2}+1|_{p}}=1.
\end{align*}
Therefore,  $g_{a,b}(x)\in\ce_{p}$ which, due to case (I), implies
$x\in A(x_{0})$.

Now let $x\notin\bq_p\setminus R$. Then $g^n_{a,b}(x)\notin K$ for
some $n\geq 0$. Hence, $g^{n+1}_{a,b}(x)\in\ce_p\subset A(x_0)$.

(III) In this step, we show that if $x\in R$, then $x\notin A(x_0)$.
Indeed, from $x\in R$, one finds that $g^n_{a,b}(x)\in K$ for all
$n\geq 0$. This means that
\begin{eqnarray}\label{BB2}
\big|(g^{n}_{a,b}(x))^{2}+1\big|_{p}\leq |b^{2}-1|_{p}<1, \ \
\forall n\in\bn.
\end{eqnarray}
By Lemma \ref{MMK1} (3), we have $g^{n}_{a,b}(x)\notin\ce_p$ for all
$n\geq 0$. Since $\ce_p$ is an open neighborhood of $x_0$, the
iterates $g^{n}_{a,b}(x)$ do not converge to $x_0$ as $n\to\infty$.
Hence, $x\notin A(x_0)$.

This completes the proof.
\end{proof}

\section{Chaoticity of the $p$-adic dynamical systems}

In this section, we prove that the renormalized dynamical system
corresponding to the model is topologically conjugate to the
symbolic shift. We show that the function $g_{a,b}$ in \eqref{gg} is
$p$-adic transitive weak repeller. Therefore, we prove that $p$-adic
dynamical system associated to $g_{a,b}$  is topologically conjugate
to a subshift of finite type.

In what follows, we always assume that the dynamical system
$g_{a,b}$ has three fixed points $\{x_0,x_1,x_2\}$ (see Theorem
\ref{g-fix}). This, due to Theorem \ref{g-fix}, means that $p\equiv
1 (mod \ 4)$. For the sake of simplicity of calculation, we suppose
that $|a-1|_p<|b-1|_p$ is satisfied.

\begin{lem}\label{cor2}
Let $r=|b-1|_p$, then one has $B_r(x_1)\cap B_r(x_2)=\emptyset$.
\end{lem}
\begin{proof} It is enough to show that $x_2\notin B_r(x_1)$. We know
that (see \eqref{solx})
$$
x_1=\frac{ab^{2}-x_0+\sqrt{\triangle}}{2}, \ \ x_2=\frac{ab^{2}-x_0-\sqrt{\triangle}}{2}.
$$
Then one gets
$$
|x_1-x_2|_p=\bigg|\frac{ab^{2}-x_0+\sqrt{\triangle}}{2}-\frac{ab^{2}-x_0-\sqrt{\triangle}}{2}\bigg|_p=|\sqrt{\triangle}|_p.
$$
Due to $\triangle=-4+p^{m+l}\varepsilon$ and $|\triangle|_p=1$, so $|\sqrt{\triangle}|_p=1$. Therefore we have $|x_1-x_2|_p=1>|b-1|_p=r$,
which means that $x_2\notin B_r(x_1)$.
\end{proof}

This lemma allows us to take the square root of $g_{a,b}$ by the
unique way on the balls $B_r(x_1)$, $B_r(x_2)$, respectively.
Therefore, one can prove the following result.

\begin{prop}\label{gg1-proposition}
The function $g_{a,b}$ given by \eqref{gg} is topologically
conjugate to
\begin{equation}\label{gg1}
k_{a,b}(x)=\left(\frac{a(b^2x+1)}{b^{2}+x}\right)^2.
\end{equation}
on the ball $B_{r}(x_1^2)$ (respectively, $B_{r}(x_2^2)$).  Here as
before $r=|b-1|_p$.
\end{prop}

\begin{proof} Let $s_1(x)=-\sqrt{x}$, $s_2(x)=\sqrt{x}$. Then using Lemma \ref{pr-g} one can prove that
$s_1$ (resp. $s_2$) homeomorphically maps $B_{r}(x_1^2)$ (resp.
$B_{r}(x_2^2)$) onto $B_{r}(x_1)$ (resp. $B_{r}(x_2)$). Moreover, we
have $g_{a,b}\circ s_i=s_i\circ k_{a,b}$ on $B_{r}(x_i^2)$, $i=1,2$.
\end{proof}

\begin{rem} Again by means of Lemma \ref{pr-g} we can establish that $K=\bar B_r(x_1)\cup \bar B_r(x_2)$, where
$\bar B_r(x_i)=\{x\in\bq_p: \ |x-x_i|_p\leq r\}$, $i=1,2$. Let
$K'=\{x\in\bq_p: \ |x^2+1|_p<1\}$. Under condition the $p\equiv 1
(mod \ 4)$, there exist $\a_1,\a_2\in\bq_p$ such that $\a_i^2=1$ for
$i=1,2$. Then, one has
$$
K=\bar B_r(\a_1)\cup \bar B_r(\a_2), \ \ K'=B_r(\a_1)\cup B_r(\a_2).
$$
Moreover, for $x,y\in B_r(\a_i)$, $i=1,2$, we have
$$
|g_{a,b}(x)-g_{a,b}(y)|_p=\frac{|x-y|_p}{r}.
$$
Hence, $g_{a,b}$ is a $p$-adic weak repeller on $K'$, but it does
not appear to be locally Lipschitz on $K\setminus K'$. Moreover,
$g_{a,b}$ on $K'$ does not yield an irreducible incidence matrix, as
described in section 2.3. Therefore, we need to work with $k_{a,b}$
instead of $g_{a,b}$ to avoid these issues. Moreover, $k_{a,b}$ is
not isometrically conjugate to $g_{a,b}$ on $K$. Indeed, one can see
that  $s(x)=x^2$ is invertible on $K$. Moreover, $s$ is an isometry
when restricted to either $\bar B_{r}(x_1)$ or $\bar B_{r}(x_2)$.
However, $s$ is not isometry on $K$ as a whole. For example, one has
\begin{eqnarray*}
|s(x_1)-s(x_2)|_p=|x_1^2-x_2^2|_p=|ab^2-x_0|_p|x_1-x_2|_p<|x_1-x_2|_p.
\end{eqnarray*}
Hence, $g_{a,b}$ is not isometrically conjugate to $k_{a,b}$ on $K$
or $B_r(x_1)\cup B_r(x_2)$. Therefore, in what follows we require
only topological conjugacy.
\end{rem}

\begin{cor}\label{gg1-proposition-cor}
One has $Fix(k_{a,b})=\{x_0^2,x_1^2,x_2^2\}$.
\end{cor}

\begin{lem}\label{lemma5-2} Let $r=|b-1|_p$. Then
$$
|k_{a,b}(x)-k_{a,b}(y)|_p=\frac{|x-y|_p}{|b-1|_p^2}, \ \ \textit{for
any}\ x,y\in B_r(x_1^2)
$$
and
$$
|k_{a,b}(x)-k_{a,b}(y)|_p=\frac{|x-y|_p}{|b-1|_p^2}, \ \ \textit{for
any}\ x,y\in B_r(x_2^2).
$$

\end{lem}
\begin{proof} Let $x,y\in B_r(x_1^2)$. Then one gets $x=x_1^2+\gamma_1$ and $y=x_1^2+\gamma_2$, where
$|\gamma_1|_p<r$ and $|\gamma_2|_p<r$.

From $x+b^2=x_1^2+b^2+\gamma_1$ we have $y+b^2=x_1^2+b^2+\gamma_2$.
Due to Lemma \ref{pr-g} (iii) one finds $|x_1^2+b^2|_p=|b-1|_p=r$
and $|x+y|_p=|2x_1^2+\gamma_1+\gamma_2|_p=1.$ Simple, but long
calculations imply that
\begin{eqnarray*}
|(b^4 +1)(x+y)+2b^2 (1+x y)|_p=|b-1|_p.\\
\end{eqnarray*}
Therefore, we obtain
\begin{eqnarray*}
|k_{a,b}(x)-k_{a,b}(y)|_{p}&=&\frac{|\left(b^4-1\right) (x-y)\left((b^4 +1)(x+y)+2b^2 (1+x y)\right)|_p}{|\left(b^2+x\right)^2 \left(b^2+y\right)^2|_p}\\
&=&\frac{|(b^4-1) (x-y)|_p|b-1|_{p}}{|\left(b^2+x\right)^2 \left(b^2+y\right)^2|_p}\\
&=&\frac{|x-y|_{p}}{|b-1|_{p}^2}.
\end{eqnarray*}
Similarly, for any $x,y\in B_r(x_2^2)$ one can show that
$$|k_{a,b}(x)-k_{a,b}(y)|_p=\frac{|x-y|_p}{|b-1|_{p}^2}.
$$
This completes the proof.
\end{proof}

\begin{lem}\label{lemma-3}
Let $X=B_r(x_1^2)\cup B_r(x_2^2),$ here as before $r=|b-1|_p$. Then
$k_{a,b}^{-1}(X)\subset X.$
\end{lem}
\begin{proof}
We show that $k_{a,b}$ has two inverse branches on the set $X$,
which are
$$
k_{1}^{-1}(x)=\left(\frac{a-b^2\sqrt{x}}{\sqrt{x}-ab^2}\right),\ \ \ k_{2}^{-1}(x)=-\left(\frac{a+b^2\sqrt{x}}{\sqrt{x}+ab^2}\right).
$$
First, let us show that for any $x\in X$, $k_{1}^{-1}(x)\in
B_{r}(x_1^2).$ Indeed, we have
\begin{eqnarray}\label{lemma-3-eq1}
|k_{1}^{-1}(x)-x_1^2|_{p}&=&\bigg|\frac{a-b^2\sqrt{x}}{\sqrt{x}-a b^2}-x_1^2\bigg|_p
=\frac{|a-b^2\sqrt{x}-x_1^2(\sqrt{x}-a b^2)|_p}{|\sqrt{x}-a b^2|_p}.
\end{eqnarray}
Now, let us compute the numerator and denominator of eq. \eqref{lemma-3-eq1}.
$$
|\sqrt{x}-a b^2|_p=|x_1+p^{m/2}\gamma-a b^2|_p=|(x_1-1)+p^{m/2}\gamma+(1-a b^2)|_p=1,
$$
\begin{eqnarray*}
|a-b^2\sqrt{x}-x_1^2(\sqrt{x}-a b^2)|_p&=&|(a-1)-(b^2-1)\sqrt{x}+(1-a b^2)-(x_1^2+1)(\sqrt{x}-a b^2)|_p\\
&<&|b-1|_p.
\end{eqnarray*}
Therefore, one gets $|k_{1}^{-1}(x)-x_1^2|_{p}<|b-1|_p$ which
implies that $k_{1}^{-1}(x)\in B_{r}(x_1^2).$

Similarly, one can show that $k_{2}^{-1}(x)\in B_{r}(x_2^2).$
Consequently, we conclude that $k_{a,b}^{-1}(X)\subset X.$
\end{proof}

\begin{lem}\label{corallary-5}
One has $B_{r}(x_i^2)\subset k_{a,b}(B_{r}(x_j^2))$, $i,j\in
\{1,2\}$.
\end{lem}
Now, we can prove main result of this section.
\begin{thm}\label{theorem-con}
Let $r=|b-1|_p$, $X=B_r(x_1^2)\cup B_r(x_2^2)$ and
$k_{a,b}:X\rightarrow \mathbb{Q}_{p}$ be a function given by
\eqref{gg1}. Then the dynamics $(J_{k_{a,b}}, k_{a,b}, |.|_p)$ is
isometrically conjugate to the shift dynamics $(\Sigma, \sigma,
d_{k_{a,b}})$.
\end{thm}

\begin{proof} It is enough to check that all conditions of Theorem \ref{exit} are satisfied. According to Lemma \ref{lemma-3}, one gets $k^{-1}_{a,b} (X)\subset X.$
By Lemma \ref{lemma5-2}, the triple $(J_{k_{a,b}},k_{a,b}, |.|_p)$ is a $p$-adic repeller.
Finally, Lemma \ref{corallary-5}, an incidence matrix $A$ has the following form
$$A=\left(
       \begin{array}{cc}
       1 & 1 \\
        1 & 1 \\
       \end{array}
     \right).
$$
Therefore, the triple $(X,k_{a,b}, \mid.\mid _p)$ is a transitive.
So, Theorem \ref{exit} implies that $p$-adic nonlinear rational
dynamical system $(J_{k_{a,b}},k_{a,b}, \mid .\mid_p)$ is
isometrically conjugate to full shift dynamical system
$(\Sigma_{A},\sigma, d_k)$. This completes the proof.
\end{proof}

It is well-known that the shift map is chaotic, and hence, from
Theorem \ref{theorem-con}, we can infer that the function \eqref{gg}
is chaotic as well. 

\begin{rem} A main aim of the present paper is to establish of the chaos for the generalized $p$-adic Ising mapping in $\bq_p$, which
is stated in Theorem \ref{theorem-con}. On the other hand, we notice
that the results in \cite{FL2} can be extended for finite extensions
of $\bq_p$. Therefore, further development of the results in this
paper can be generalized as well for finite extensions, which will
be a topic our next works.
\end{rem}

\section{An application: Gibbs measures for the $p$-adic Ising-Vannimenus model}

In this section, we study the existence of periodic $p$-adic Gibbs
measures of $p$-adic Ising-Vannimenus model given in \cite{MDA}. In
\cite{MSK1} it has been studied 2-periodic $p$-adic Gibbs measures
for the Ising-Vannimenus model on the Cayley tree of order two. In
this section, we show that there are many kinds of periodic $p$-adic
Gibbs measures (for definitions we refer to the appendix).

Let us consider a $H_m$-periodic function $\h=\{\h_x\}_{x\in
V\setminus\{x^0\}}$. From the $H_m$-periodicity we infer that
there is a $m$-collection of vectors  $\{\h_0,\dots,\h_{m-1}\}$,
such that $\h_x=\h_i$, if $d(x,x^0)\equiv i(\textrm{mod}\ m)$,
$i=0,\dots,m-1$. On the invariant line, we have
$\h_i=(h_i,1,\dots,1)$ $(i=0,\dots,m-1)$.

Then the equation \eqref{canonic3} for the $H_m$-periodic functions
reduces to the following system
\begin{equation}\label{eq12}
h_i=g_{a,b}(h_{i+1}), \ \ h_m=g_{a,b}(h_{0}), \ i=1,\dots,m-1.
\end{equation}
It is clear that the equation \eqref{eq12} is equivalent to
finding $m$-periodic points of the function $g_{a,b}$. Hence, the
existence of periodic orbits of the function implies the existence
of $H_m$-periodic $p$-adic quasi Gibbs measures.

It is well-known that the shift operator has infinitely many
periodic points, therefore, Theorem  \ref{theorem-con} implies that
the function $k_{a,b}$ also has infinitely many periodic points.
Also, because $g_{a,b}$ is topologically conjugate to $k_{a,b}$, the
function $g_{a,b}$ also has infinitely many periodic points. Hence,
there are infinitely many $H_m$-periodic $p$-adic quasi Gibbs
measures for the $p$-adic Ising-Vannimenus model.

\section*{Acknowledgments} The authors would like to thank the referees for
their useful suggestions which allowed to improve the content of the
paper.

\section{Appendix}
\subsection{$p$-adic measure}

Let $(X,\cb)$ be a measurable space, where $\cb$ is an algebra of
subsets $X$. A function $\m:\cb\to \bq_p$ is said to be a {\it
$p$-adic measure} if for any $A_1,\dots,A_n\in\cb$ such that
$A_i\cap A_j=\emptyset$ ($i\neq j$) the equality holds
$$
\mu\bigg(\bigcup_{j=1}^{n} A_j\bigg)=\sum_{j=1}^{n}\mu(A_j).
$$

A $p$-adic measure is called a {\it probability measure} if
$\mu(X)=1$.  One of the important conditions is boundedness, namely
a $p$-adic probability measure $\m$ is called {\it bounded} if
$\sup\{|\m(A)|_p : A\in \cb\}<\infty $. For more detailed
information about $p$-adic measures we refer to
\cite{AKh},\cite{K3},\cite{KhN}.

\subsection{Cayley tree}

Let $\Gamma^k_+ = (V,L)$ be a semi-infinite Cayley tree of order
$k\geq 1$ with the root $x^0$ (where each vertex has exactly $k+1$
edges, except for the root $x^0$, which has $k$ edges). Here $V$ is
the set of vertices and $L$ is the set of edges. The vertices $x$
and $y$ are called {\it nearest neighbors} and they are denoted by
$l=<x,y>$ if there exists an edge connecting them. A collection of
the pairs $<x,x_1>,\dots,<x_{d-1},y>$ is called a {\it path} from
the point $x$ to the point $y$. The distance $d(x,y), x,y\in V$, on
the Cayley tree, is the length of the shortest path from $x$ to $y$.

Recall a coordinate structure in $\G^k_+$:  every vertex $x$
(except for $x^0$) of $\G^k_+$ has coordinates $(i_1,\dots,i_n)$,
here $i_m\in\{1,\dots,k\}$, $1\leq m\leq n$ and for the vertex
$x^0$ we put $(0)$.  Namely, the symbol $(0)$ constitutes level 0,
and the sites $(i_1,\dots,i_n)$ form level $n$ ( i.e.
$d(x^0,x)=n$) of the lattice.


For $x\in \G^k_+$, $x=(i_1,\dots,i_n)$ put
\begin{equation}\label{S(x)}
 S(x)=\{(x,i):\ 1\leq
i\leq k\},
\end{equation}
here $(x,i)$ is short for $(i_1,\dots,i_n,i)$. This set is called
a set of {\it direct successors} of $x$.

Let us define on $\G^k_+$ a binary operation
$\circ:\G^k_+\times\G^k_+\to\G^k_+$ as follows: for any two
elements $x=(i_1,\dots,i_n)$ and $y=(j_1,\dots,j_m)$ put
\begin{equation}\label{binar1}
x\circ
y=(i_1,\dots,i_n)\circ(j_1,\dots,j_m)=(i_1,\dots,i_n,j_1,\dots,j_m)
\end{equation}
and
\begin{equation}\label{binar2}
x\circ x^0=x^0\circ x= (i_1,\dots,i_n)\circ(0)=(i_1,\dots,i_n).
\end{equation}

By means of the defined operation $\G^k_+$ becomes a
noncommutative semigroup with a unit. Using this semigroup
structure one defines translations $\tau_g:\G^k_+\to \G^k_+$,
$g\in \G^k_+$ by
\begin{equation}\label{trans1}
\tau_g(x)=g\circ x.
\end{equation}
It is clear that $\tau_{(0)}=id$.

Let $H\subset \G^k_+$ be a sub-semigroup of $\G^k_+$ and
$h:\G^k_+\to Y$ be a $Y$-valued function defined on $\G^k_+$. We say
that $h$ is {\it $H$-periodic} if $h(\tau_g(x))=h(x)$ for all $g\in
H$ and $x\in \G^k_+$. Any $\G^k_+$-periodic function is called {\it
translation invariant}. For each $m\geq 2$ we put
\begin{equation}\label{sub}
H_m=\{x\in \G^k_+: \ d(x,x^0)\equiv 0 (\textrm{mod}\ m) \}.
\end{equation}
One can check that $H_m$ is a sub-semigroup.

Let us set
$$ W_n=\{x\in V| d(x,x^0)=n\}, \ \ \
V_n=\bigcup_{m=1}^n W_m, \ \ L_n=\{l=<x,y>\in L | x,y\in V_n\}.
$$

Two vertices $x,y\in V^0$ are called the \textit{next-nearest
neighbors} if $d(x,y)=2$. The next-nearest-neighbors vertices $x$
and $y$ are called the \textit{prolonged next-nearest neighbors}
if $x\in W_{n-2}$ and $y\in W_n$ for some $n\geq 1$, which are
denoted by $>x,y<$. The next-nearest-neighbor vertices $x$ and $y$
are called \textit{one-level next-nearest-neighbors} if $x,y\in
W_n$ for some $n$ and they are denoted by $>\overline{x,y}<$.

\subsection{$p$-adic Ising-Vannimenus (IV) model and its $p$-adic Gibbs measures}

In this subsection we consider the $p$-adic Ising-Vannimenus model
given in \cite{MDA} where spin takes values in the set
$\Phi=\{-1,+1\}$, ($\Phi$ is called a {\it state space}) and is
assigned to the vertices of the tree $\G^k_+=(V,L)$. A configuration
$\s$ on $V$ is then defined as a function $x\in V\to\s(x)\in\Phi$;
in a similar manner one defines configurations $\s_n$ and $\w$ on
$V_n$ and $W_n$, respectively. The set of all configurations on $V$
(resp. $V_n$, $W_n$) coincides with $\Omega=\Phi^{V}$ (resp.
$\Omega_{V_n}=\Phi^{V_n},\ \ \Omega_{W_n}=\Phi^{W_n}$). One can see
that $\Om_{V_n}=\Om_{V_{n-1}}\times\Om_{W_n}$. Using this, for given
configurations $\s_{n-1}\in\Om_{V_{n-1}}$ and $\w\in\Om_{W_{n}}$ we
define their concatenations  by
$$
(\s_{n-1}\vee\w)(x)= \left\{
\begin{array}{ll}
\s_{n-1}(x), \ \ \textrm{if} \ \  x\in V_{n-1},\\
\w(x), \ \ \ \ \ \ \textrm{if} \ \ x\in W_n.\\
\end{array}
\right.
$$
It is clear that $\s_{n-1}\vee\w\in \Om_{V_n}$.

The Hamiltonian $H_n:\Om_{V_n}\to\bq_p$ of the $p$-adic
Ising-Vannimenus model has a form

\begin{equation}\label{ham1}
H_{n}(\sigma)=J\sum\limits_{< x,y>\in
L_{n}}\sigma(x)\sigma(y)+J_{1}\sum\limits_{>x,y<: x,y\in
V_{n}}\sigma(x)\sigma(y)+J_{0}\sum\limits_{>\overline{x,y}<: x,y\in
V_{n}}\sigma(x)\sigma(y)
\end{equation}
where $J_{1},J,J_0\in B(0,p^{-1/(p-1)})$ are coupling constants.

Note that the last condition together with the strong triangle
inequality implies the existence of $\exp_p(H_n(\s))$ for all
$\s\in\Om_{V_n}$, $n\in\bn$. This is required to our construction.


Assume that
$\h:{V\setminus\{x^{(0)}\}}\times{V\setminus\{x^{(0)}\}}\to\bq_p^{\Phi\times\Phi}$
be a mapping i.e. $\{x,y\}\to \h_{xy}$, where
$\h_{xy}=(h_{xy,--},h_{xy,-+},h_{xy,+-},h_{xy,++})$,
$h_{xy,\pm\pm}\in\bq_p$ and $x,y\in V\setminus\{x^{(0)}\}$.

Given $n\in\bn$, let us consider a $p$-adic probability measure
$\m^{(n)}_\h$ on $\Om_{V_n}$ defined by

\begin{equation}\label{measure1}
\m_\h^{(n)}(\sigma)=\frac{1}{Z_{n}^{(\h)}}{\exp_{p}}{\left(H_{n}(\sigma)\right)}{\underset{x\in
W_{n-1},y\in
S(x)}{\prod}\left(h_{xy,\sigma(x)\sigma(y)}\right)^{\sigma(x)\sigma(y)}}
\end{equation}
Here, $\s\in\Om_{V_n}$, and $Z_n^{(\h)}$ is the corresponding
normalizing factor called a {\it partition function} given by

\begin{equation}\label{partition}
Z_{n}^{(\h)}=\sum_{\s\in\Omega_{V_n}}{\exp_{p}}{\left(H_{n}(\sigma)\right)}{\underset{x\in
W_{n-1},y\in
S(x)}{\prod}\left(h_{xy,\sigma(x)\sigma(y)}\right)^{\sigma(x)\sigma(y)}}.
\end{equation}

We recall \cite{M13} that one of the central results of the theory
of probability concerns a construction of an infinite volume
distribution with given finite-dimensional distributions, which is a
well-known {\it Kolmogorov's extension Theorem} \cite{Sh}. In this
paper we apply the  Kolmogorov's Theorem in a $p$-adic context
\cite{KL}. Namely, a $p$-adic probability measure $\m$ on $\Om$,
which is compatible with defined ones $\m_\h^{(n)}$, i.e.
\begin{equation}\label{CM}
\m(\s\in\Om: \s|_{V_n}=\s_n)=\m^{(n)}_\h(\s_n), \ \ \ \textrm{for
all} \ \ \s_n\in\Om_{V_n}, \ n\in\bn,
\end{equation}
exists if the measures $\m_\h^{(n)}$, $n\geq 1$ satisfy the {\it
compatibility condition}, i.e.
\begin{equation}\label{comp}
\sum_{\w\in\Om_{W_n}}\m^{(n)}_\h(\s\vee\w)=\m^{(n-1)}_\h(\s), \ \
\textrm{for all} \ \ \s\in\Om_{V_{n-1}}.
\end{equation}

We should stress that using the compatibility condition for the
Ising model on the Bethe lattice, in the real case (see \cite{Roz}
for review).

Following \cite{M13} if for some function $\h$ the measures
$\m_\h^{(n)}$ satisfy the compatibility condition, then there is a
unique $p$-adic probability measure, which we denote by $\m_\h$,
since it depends on $\h$. Such a measure $\m_\h$ is said to be {\it
a $p$-adic quasi Gibbs measure} corresponding to the $p$-adic
Ising-Vannimenus model. By $Q\cg(H)$ we denote the set of all
$p$-adic quasi Gibbs measures associated with functions $\h=\{\h_x,\
x\in V\}$. If there are at least two distinct $p$-adic quasi Gibbs
measures $\m,\n\in Q\cg(H)$ such that $\m$ is bounded and $\n$ is
unbounded, then we say that {\it a phase transition} occurs. By
another words, one can find two different functions $\sb$ and $\h$
defined on $\bn$ such that there exist the corresponding measures
$\m_\sb$ and $\m_\h$, for which one is bounded, another one is
unbounded. Moreover, if there is a sequence of sets $\{A_n\}$ such
that $A_n\in\Om_{V_n}$ with $|\m(A_n)|_p\to 0$ and
$|\n(A_n)|_p\to\infty$ as $n\to\infty$, then we say that there
occurs a {\it strong phase transition}.

Now one can ask for what kind of functions $\h$ the measures
$\m_\h^{(n)}$ defined by \eqref{measure1} would satisfy the
compatibility condition \eqref{comp}. The following theorem gives
an answer to this question.

\begin{thm}\label{compatibility}\cite{MDA}
The measures $\m^{(n)}_\h$, $ n=1,2,\dots$ (see \eqref{measure1})
satisfy the compatibility condition \eqref{comp} if and only if for
any $n\in \bn$ the following equation holds:

\begin{equation}\label{canonic3}
\begin{cases}
h_{xy,++}\cdot h_{xy,-+}=\prod\limits_{z\in S(y)}
\frac{(ab)^2h_{yz,++}h_{yz,+-}+1}{a^2h_{yz,++}h_{yz,+-}+b^2}\\
h_{xy,--}\cdot h_{xy,+-}=\prod\limits_{z\in S(y)}
\frac{(ab)^2h_{yz,--}h_{yz,-+}+1}{a^2h_{yz,--}h_{yz,-+}+b^2}\\
 h_{xy,++}\cdot h_{xy,+-}=\prod\limits_{z\in S(y)}
\frac{\big((ab)^2h_{yz,++}h_{yz,+-}+1\big)h_{yz,-+}}{\big(a^2h_{yz,--}h_{yz,-+}+b^2\big)h_{yz,+-}}\\
\end{cases}
\end{equation}
where $a=\exp_p(J)$, $b=\exp_p(J_1)$.
\end{thm}

\end{document}